\documentclass[10pt]{amsart}

\usepackage[utf8]{inputenc}
\usepackage[T1]{fontenc}
\usepackage[pdftex]{graphicx}
\usepackage{latexsym}
\usepackage{amsmath,amssymb,amsthm,amsfonts}
\usepackage{mathtools}
\usepackage{enumitem}
\usepackage{faktor}
\usepackage{enumerate}
\usepackage{hyperref}
\usepackage{tikz}
\usepackage{xcolor}

\newtheorem{theorem}{Theorem}[section]
\newtheorem{corollary}[theorem]{Corollary}
\newtheorem{proposition}[theorem]{Proposition}
\newtheorem{lemma}[theorem]{Lemma}
\theoremstyle{definition}

\newtheorem*{remark}{Remark}

\newtheorem{problem}[theorem]{Problem}
\newtheorem{conjecture}[theorem]{Conjecture}

\newcommand{\K}{\mathcal{K}}

\newcommand{\Hmm}[1]{\leavevmode{\marginpar{\tiny%
			$\hbox to 0mm{\hspace*{-0.5mm}$\leftarrow$\hss}%
			\vcenter{\vrule depth 0.1mm height 0.1mm width \the\marginparwidth}%
			\hbox to 0mm{\hss$\rightarrow$\hspace*{-0.5mm}}$\\\relax\raggedright #1}}}
	
	\newcommand{\eat}[1]{}

\begin{document}
	\author[P. Bartmann, M. Keller]{Philipp Bartmann, Matthias Keller}
	
	\address{Philipp Bartmann: Institut f\"ur Mathematik, Universit\"at Potsdam
		14476  Potsdam, Germany}
	\email{philipp.bartmann@uni-potsdam.de}
    \address{Matthias Keller:   Israel Institute of Advanced Studies, Jerusalem, Israel; Institut f\"ur Mathematik, Universit\"at Potsdam
14476  Potsdam, Germany}
	\email{matthias.keller@uni-potsdam.de}

	\title[Eigenvalue growth of the discrete Hodge Laplacian across dimensions]{Eigenvalue growth of the discrete Hodge Laplacian across dimensions}

\begin{abstract}
We prove several bounds on the largest and smallest eigenvalues of the combinatorial Hodge Laplacian $\Delta^H_k$ of a finite simplicial complex $\Sigma.$ As a consequence, we obtain new vanishing criteria for cohomology groups $H^k(\Sigma,\mathbb{R)}$ and confirm a conjecture of O on the dimensional monotonicity of the largest eigenvalue.
\end{abstract}
\maketitle

\section{Introduction}

Starting with the work of Eckmann \cite{Eckmann1944}, the study of combinatorial simplicial complexes and their Laplacians has been a fruitful area of mathematics. In particular, the spectrum of the discrete Hodge Laplacian $\Delta^H_k = \delta_{k-1}\partial_{k-1}+\partial_{k}\delta_k$ turned out to be of tremendous importance, as it encodes geometric information about the complex and determines its homology. Consequently, a large body of research is devoted to the interplay between spectrum and geometry; see \cite{AharoniBergerMeshulam,SO2,DuvRei,GundertWagner,HJ,lewfaces,lew2025eigenvalue,oppenheim_local,ShuklaYogeshwaran,zhang2026upper} to name a few.
In a recent paper \cite{HelmholzianGraph} on the Helmholtzian $\Delta^H_1$ of a graph $G$, which one interprets as a one-dimensional simplicial complex, the authors notice that the largest eigenvalue $\lambda^H_1$ of $\Delta^H_1$ seems to always agree with the largest eigenvalue $\lambda(G)$ of the graph Laplacian $L(G).$ In their paper, the authors formulate the following problem.
\begin{problem}\label{problem}
Is $\lambda^H_1 = \lambda(G)$ for every graph $G$?
\end{problem}
Motivated by this, the authors of \cite{SO2} recently proved the estimate
\[\lambda^H_1\leq \lambda(G)+\frac{1}{3}(n-\lambda(G)),\] where $n$ denotes the number of vertices in $G.$ They then formulate a more general conjecture, initially posed by O in an earlier version of their paper (\cite{SO1}), that generalizes Problem~\ref{problem} to arbitrary simplicial complexes. That conjecture asks about the largest eigenvalue $\lambda^+_k$ of $\Delta^+_k = \partial_k\delta_k.$

\begin{conjecture}\label{Ocon}
    For every $k\geq 1$, one has $\lambda^+_k\leq\lambda^+_{k-1}.$
\end{conjecture}

The authors of \cite{SO2} proceed to prove the conjecture for various special cases and obtain the estimate
\[\lambda^+_k\leq \lambda^+_{k-1} + \frac{(n-\lambda^+_{k-1})}{k+2}\]
for the general case.

The main result of this paper proves Conjecture~\ref{Ocon} in full generality and thereby also resolves Problem~\ref{problem}. Furthermore, we generalize an estimate for flag complexes by Aharoni-Berger-Meshulam \cite{AharoniBergerMeshulam} to general simplicial complexes and obtain new vanishing criteria for the reduced cohomology groups of a simplicial complex $\Sigma.$ These results also generalize those obtained by Lew in \cite{lewgaps} about large missing faces in $\Sigma.$ The general idea behind our result is to view $\Sigma$ as a sub-complex of the \textit{complete} complex $\mathcal{K},$ while keeping control of the error term $R_k$ that occurs when restricting the boundary operator $\widehat\delta_k$ on $\mathcal{K}$ to the boundary operator $\delta_k$ on $\Sigma.$\\
The paper is structured as follows: In Section~\ref{prelim} we recall the setting of finite  simplicial complexes and their Laplacians, following the formalism from \cite{BK}. In Section~\ref{local} we recall the localization method from \cite{AharoniBergerMeshulam}, which allows us to compare eigenvalues in different dimensions. There, we also slightly adapt the method to fit our setting. Such localization techniques were introduced by Garland in \cite{Garland} and are well established in the literature. In Section~\ref{mainsection} we prove our main result, which estimates the largest (or smallest) eigenvalues $\lambda^H_k$ of the Laplacian $\Delta^H_k$ in terms of the largest (or smallest) eigenvalue $\lambda_{k-1}^H$ and the number of vertices $n.$ As a corollary, we obtain a vanishing criterion for the reduced cohomology groups $H^k(\Sigma,\mathbb{R}).$ Our result can be seen as a generalization of that in \cite{AharoniBergerMeshulam} and is similar in nature to those in \cite{ShuklaYogeshwaran}. In Section~\ref{morebounds} we prove some more estimates on the spectrum of $\Delta^H_k$ that we derive from lower bounds on the Forman curvature $c^H$ or the number of free faces in $\Sigma.$ These estimates are independent of the main result and will be used in Section~\ref{alex}, where we apply our results to the Alexander dual $\Sigma^*$ and interpret the results in terms of $\Sigma.$

\section{Preliminaries}\label{prelim}
Let $V$ be a finite set of vertices. A subset $\Sigma$ of the power set $ \mathcal{P}(V)$ of $V$ is called an \textit{(abstract) simplicial complex} if it is closed under taking subsets, i.e., if for all $\sigma\in \Sigma$ and $\tau\subseteq\sigma$ one also has $\tau\in\Sigma.$ The elements of $\Sigma$ are called \textit{simplices} and we refer to the set of simplices of cardinality $k+1$ as $\Sigma_k.$ The index $k$ is understood as the \textit{dimension} of a simplex and we call the quantity
\[\dim(\Sigma) = \max\lbrace k\mid\Sigma_k\neq \emptyset\rbrace\] the \textit{dimension} of $\Sigma.$ Moreover, whenever $\sigma\in\Sigma$ and $\tau\subseteq\sigma$ with $\#\sigma\backslash\tau = 1$ we write $\tau\prec\sigma$ or $\sigma\succ\tau$ and say that $\tau$ is a \textit{face} of $\sigma$ or that $\sigma$ is a \textit{coface} of $\tau.$ For any $\tau\in\Sigma$ and $v\in V\backslash\tau$ such that $\sigma = \lbrace v\rbrace\cup\tau\in\Sigma$, we use the notation $\sigma = v\tau.$\\

Let $C(\Sigma_k)$ be the space of real (or complex) valued functions over $\Sigma_k.$ We turn this into the Hilbert space $\ell^2(\Sigma_k)$ by equipping it with the scalar product and  the norm
$$\langle \phi,\psi\rangle =\sum_{\tau\in\Sigma_k}\phi(\tau)\overline{\psi(\tau)},\qquad \Vert\phi\Vert^2 = \sum_{\tau\in\Sigma_k}\vert\phi(\tau)\vert^2.$$

Next, we define \textit{coboundary operators} $\delta_k:\ell^2(\Sigma_k)\rightarrow\ell^2(\Sigma_{k+1})$ for $k\geq -1$ as linear operators that satisfy \[\delta_{k+1}\delta_k = 0\] and $\delta_k1_\tau(\sigma)\in\lbrace\pm1,0\rbrace$ for all $\tau\in\Sigma_k,\sigma\in\Sigma_{k+1},$ with $\delta_k1_\tau(\sigma)\neq0$ iff $\tau\prec\sigma.$
\begin{remark}
    Our definition for a coboundary operator is different to the one in the literature; cf. \cite{AharoniBergerMeshulam,HJ,Parzanchevski,RT,zhang2026upper}. Here, our coboundary operator acts on functions, rather than alternating forms over oriented simplices. However, one can show that our definition is equivalent to the usual one, up to a unitary transformation; cf. \cite[Section 2.2]{BK}. This means, in particular, that the spectrum of Laplace operators remains unchanged. Our approach will allow us to track sign changes that naturally occur when considering alternating forms, more efficiently in order to prove certain identities.
\end{remark}

For later reference, we refer to the term $\delta_k1_\tau(\sigma)$ as $\theta(\tau,\sigma),$ whenever $\tau\prec\sigma.$ That way, one also sees that $\delta_k$ acts as 
\[\delta_k\phi(\sigma) = \sum_{\tau\prec\sigma}\theta(\tau,\sigma)\phi(\tau).\]
Given $\delta_k,$ we define the \textit{boundary operator} $\partial_k:\ell^2(\Sigma_{k+1})\rightarrow\ell^2(\Sigma_{k})$ as the adjoint of $\delta_k,$ namely $\partial_k = \delta_k^*.$ One easily finds that $\partial_k$ acts as

\[\partial_k\phi(\rho) = \sum_{\tau\succ\rho}\theta(\rho,\tau)\phi(\tau).\]

With boundary and coboundary operators at hand, we define the operators

\[\Delta^+_k = \partial_{k}\delta_k,\quad\Delta^-_k = \delta_{k-1}\partial_{k-1},\quad\Delta^H_k = \Delta^+_k+\Delta^-_k\]
and call them the $k$\textit{-up-}, $k$\textit{-down-} and $k$\textit{-Hodge-Laplacian} respectively.
For convenience, we make a certain choice for the coboundary operators in relation to the complete simplicial complex. More precisely, we consider the \textit{complete} simplicial complex $\K$ over $V,$ which is simply the powerset of $V$ \[\K = \mathcal{P}(V).\] We then choose coboundary and boundary operators $\widehat\delta_k$ and $\widehat\partial_k$ for $\K.$ 
Since $\Sigma\subseteq \K,$ we can identify $\ell^2(\Sigma_k)$ as a subspace of $\ell^2(\K_k).$ To do so, let $\pi_k:\ell^2(\K_k)\rightarrow\ell^2(\Sigma_k)$ be the restriction of functions in $\ell^2(\K_k)$ to $\ell^2(\Sigma_k).$ Likewise, let $\iota_k:\ell^2(\Sigma_k)\rightarrow\ell^2(\K_k)$ be the inclusion of $\ell^2(\Sigma_k)$ into $\ell^2(\K_k),$ obtained by extending functions on $\Sigma_k$ to $\K_k$ by zero.
\\
Then $\delta_k = \pi_{k+1}\widehat\delta_k\iota_k$ define coboundary operators for $\Sigma$ and $\partial_k = \pi_k\widehat\partial_k\iota_{k+1}$ are the respective boundary operators for all $k$. The reason for  this particular choice is the following lemma.
\begin{lemma}\label{nId}
    Let $n = \#V$, $k\geq 0$ and $\phi\in\ell^2(\Sigma_k)$. Then, 
    \[n\Vert\phi\Vert^2 = \Vert\delta_k\phi\Vert^2+\Vert\partial_{k-1}\phi\Vert^2+\Vert R_k\phi\Vert^2,\]
    where $R_k = \widehat\delta_k\iota_k-\iota_{k+1}\delta_k.$
\end{lemma}
\begin{proof}
    One easily verifies that \[\Vert\phi\Vert^2 = \Vert\iota_k\phi\Vert^2,\quad\Vert\delta_k\phi\Vert^2+\Vert R_k\phi\Vert^2 = \Vert\widehat\delta_k\iota_k\phi\Vert^2\quad\text{and}\quad\Vert\partial_{k-1}\phi\Vert^2 = \Vert\widehat\partial_{k-1}\iota_k\phi\Vert^2.\]
    Thus, the statement is equivalent to 
    \[n\Vert\psi\Vert^2 = \Vert\widehat\delta_k\psi\Vert^2+\Vert\widehat\partial_{k-1}\psi\Vert^2,\]
    with $\psi = \iota_k\phi.$
    However, this is well-known to be true, because the Hodge Laplacian $\widehat\Delta_k^H$ on $\K$ satifies \[\widehat\Delta^H_k = n\cdot \mathrm {Id}_{\K_k}\] and 
    \[ \Vert\widehat\delta_k\psi\Vert^2+\Vert\widehat\partial_{k-1}\psi\Vert^2 = \langle\widehat\Delta^H_k\psi,\psi\rangle.\]
    The second  identity follows from the definition and the first formula is  folklore and is mentioned for example in \cite[Lemma~8]{GundertWagner}. Alternatively, it can also be verified with an easy computation using the Schrödinger operator representation of the Hodge Laplacian from \cite[Lemma~3.13]{BK} or a direct proof is found in \cite[Lemma~3.2]{zhang2026upper}.
\end{proof}
\section{Localization}\label{local}
In this section, we recall the localization technique for functions from \cite{AharoniBergerMeshulam}. We also slightly adapt it to our setting.\\

Let $\phi\in\ell^2(\Sigma_k)$ and fix a vertex $v\in V.$ 
We  define the function $\phi_v\in\ell^2(\Sigma_{k-1})$ for $\rho\in\Sigma$ such that $v\rho\in \Sigma$ as
\[\phi_v(\rho) = \phi(v\rho)\theta(\rho,v\rho)\]
and $0$ otherwise.

The following identities were first shown in \cite{AharoniBergerMeshulam}. Since our setting is slightly different, we include a proof for the convenience of the reader.

\begin{lemma}\label{claims}
    Let $\phi\in\ell^2(\Sigma_k),~k\geq 0.$ Then,
    \[(k+1)\Vert\phi\Vert^2 = \sum_{v\in V}\Vert\phi_v\Vert^2\quad\] 
    and for $k\geq 1,$ \[\quad k\Vert\partial_{k-1}\phi\Vert^2 = \sum_{v\in V}\Vert\partial_{k-2}\phi_v\Vert^2.\]
\end{lemma}
\begin{proof}
    The first identity follows as 
    \begin{align*}
        \sum_{v\in V}\Vert\phi_v\Vert^2 = \sum_{v\in V}\sum_{\substack{\rho\in\Sigma_{k-1}\\v\rho\in\Sigma_k}}\vert\phi(v\rho)\vert^2 = (k+1)\sum_{\tau\in\Sigma_k}\vert\phi(\tau)\vert^2 = (k+1)\Vert\phi\Vert^2.
    \end{align*}
    The factor $(k+1)$ arises, because for every $\tau\in\Sigma_k$ there are exactly $k+1$ tuples $(v,\rho)\in V\times\Sigma_{k-1}$ such that $\tau = v\rho.$
    The second equality follows in a similar way: For any $\tau' = v\rho$ with $v\in V$, $\rho\in\Sigma_{k-2}$ and any $\tau\succ\rho$ such that $v\tau\in\Sigma_k$, one finds that 
    \begin{equation}\label{thetas}
        0 = \delta_{k-1}\delta_{k-2} 1_\rho(v\tau) = \theta(\rho,\tau')\theta(\tau',v\tau)+\theta(\rho,\tau)\theta(\tau,v\tau).
    \end{equation}
    Thus,
    \begin{align*}
        \sum_{v\in V}\Vert\partial_{k-2}\phi_v\Vert^2 &=\sum_{v\in V} \sum_{\rho\in\Sigma_{k-2}}\Big\vert\sum_{\substack{\tau\succ\rho\\v\tau\in\Sigma_k}}\theta(\rho,\tau)\theta(\tau,v\tau)\phi(v\tau)\Big\vert^2\\
        & = \sum_{v\in V}\sum_{\rho\in\Sigma_{k-2}}\Big\vert\sum_{\substack{\tau\succ\rho\\v\tau\in\Sigma_k}}\theta(\rho,v\rho)\theta(v\rho,v\tau)\phi(v\tau)\Big\vert^2\\
        & = k\sum_{\tau'\in\Sigma_{k-1}}\Big\vert\sum_{\sigma\succ\tau'}\theta(\tau',\sigma)\phi(\sigma)\Big\vert^2\\
        &=k\Vert\partial_{k-1}\phi\Vert^2.
    \end{align*}
In the second equality, we used (\ref{thetas}) and in the third equality, we identified $\tau' = v\rho$ and $\sigma = v\tau.$
\end{proof}

\section{Main Result}\label{mainsection}

We prove the following theorem, which shows that the maximal eigenvalues of $\Delta^+_k,\Delta_k^-,\Delta^H_k$ do not increase in $k.$ This confirms a conjecture of O in a recent paper \cite[Conjecture~1.1]{SO2}. We also generalize the results for flag complexes from \cite{AharoniBergerMeshulam} and those about missing faces from \cite{lewgaps}.

Let 
\[A_k = \lbrace\sigma\in \K_{k+1}\backslash\Sigma_{k+1}\mid\mbox{there is }\tau\in\Sigma_k \mbox{ such that }\tau\subseteq\sigma\rbrace\] and define

\[\alpha_k = \alpha_k(\Sigma) = \min_{\sigma\in A_k}\#\lbrace\tau\in\Sigma_k\mid\tau\subseteq\sigma\rbrace -1.\]
Similarly, define
\[\beta_k = \beta_k(\Sigma) = \max_{\sigma\in A_k}\#\lbrace\tau\in\Sigma_k\mid\tau\subseteq\sigma\rbrace -1.\]
If $A_k = \emptyset,$ then we define $\alpha_k = \beta_k = 0.$\\
The constants $\alpha_k$ and $\beta_k$ describe how many faces in $\Sigma$ a \textit{non-simplex} of $\Sigma$ can have.
One clearly has $0\leq\alpha_k\leq\beta_k\leq k+1.$\\
We denote the spectrum of an operator $A$ by $\sigma(A).$
\begin{theorem}[Main estimates]\label{main}
    Let $\Sigma$ be a simplicial complex with $n = \#V.$ Let $\circ\in\lbrace\pm,H\rbrace$ and $\lambda^\circ_k = \max\sigma(\Delta^\circ_k), \mu^H_k= \min\sigma(\Delta^H_k)$ for $1\leq k\leq\dim(\Sigma).$ Then,
    \[(k+1-\alpha_k)\lambda^\circ_k+\alpha_kn\leq(k+1)\lambda^\circ_{k-1}\]
    and 
    \[(k+1-\beta_k)\mu^H_k+\beta_kn\geq(k+1)\mu^H_{k-1}.\]
\end{theorem}

The second estimate generalizes an inequality by Ahanroni-Berger-Meshulam  \cite[Theorem~1.1]{AharoniBergerMeshulam} and its generalization by Lew \cite[Theorem~1.2]{lewgaps}. This is discussed in detail at the end of the section.\\
The first estimate above resolves O's conjecture \cite[Conjecture~1.1]{SO2}.

\begin{corollary}[O's conjecture]\label{conjecture}
    Let $\Sigma$ be an abstract simplicial complex with $n = \#V.$ Let $\circ\in\lbrace\pm,H\rbrace$ and $\lambda^\circ_k = \max\sigma(\Delta^\circ_k)$ for $1\leq k\leq\dim(\Sigma).$ Then
    \[\lambda^\circ_k\leq\lambda^\circ_{k-1}.\]
    Furthermore, $\lambda^H_k = \lambda^-_k.$  
\end{corollary}
\begin{proof}
    The bound from Theorem~\ref{main} implies
    \[(k+1)\lambda^\circ_k+\alpha_k(n-\lambda^\circ_k)\leq(k+1)\lambda^\circ_{k-1}.\]
    The first statement now follows because $\lambda^\circ_k\leq n$ by Lemma~\ref{nId},  see also \cite{DuvRei}.\\
    The second statement follows from the first and from the well-known spectral identities for $k\geq0$
    \[\sigma(\Delta^H_k)\backslash\lbrace0\rbrace = (\sigma(\Delta^-_k)\cup\sigma(\Delta^+_k))\backslash\lbrace0\rbrace\quad\text{and}\quad\sigma(\Delta^+_k)\backslash\lbrace 0\rbrace = \sigma(\Delta^-_{k+1})\backslash\lbrace 0\rbrace.\]
    Thus, $\lambda_k^H =\max\lbrace\lambda_k^-,\lambda_{k}^+\rbrace = \max\lbrace\lambda_{k-1}^+,\lambda_{k}^+\rbrace=\lambda_{k-1}^+=\lambda_k^-$, where we used $\lambda_k^+\leq\lambda_{k-1}^+$ in the third equality.\end{proof}

\begin{remark}
    The statement about $\mu^H_k$ in Theorem~\ref{main} generalizes the estimate from \cite[Theorem~1.1]{AharoniBergerMeshulam} for flag complexes to arbitrary finite simplicial complexes. In fact, for a flag complex one finds that $\beta_k\leq 1,$ in which case we recover the estimate in \cite[Theorem~1.1]{AharoniBergerMeshulam}. The fact $\beta_k\leq 1$ follows because otherwise there exists a $\sigma\in A_k$ with at least three simplices $\tau_1,\tau_2,\tau_3\in\Sigma_k$ such that $\tau_i\subseteq\sigma,i = 1,2,3.$ However, three such simplices necessarily form a clique, which implies $\sigma\in\Sigma_{k+1},$ contradicting the fact that $\sigma\in A_k.$ Our result is also similar to an estimate due to Lew \cite{lewgaps}, but holds in a more general setting that does not require large missing faces in $\Sigma.$
\end{remark}

Before we prove Theorem \ref{main}, we will need some lemmas. Recall the definition $R_k = \widehat\delta_k\iota_k-\iota_{k+1}\delta_k.$

\begin{lemma}\label{R_k-R_k-1}
    Let $\sigma\in \K\backslash \Sigma,$ $ v\in V,$ $\tau'\in \K_k\backslash \Sigma_k,$ $k\geq 1$ and $\sigma = v\tau'.$ Then, for any $\phi\in C(\Sigma_k)$, we have
    \[\vert R_{k}\phi(\sigma)\vert = \vert R_{k-1}\phi_v(\tau')\vert.\]
\end{lemma}
\begin{proof}
    First of all, note that for any $\tau\in\Sigma_k$ with $\tau\subseteq\sigma,$ we have
    $v\in\tau$ since otherwise $\tau = \tau'\notin\Sigma_k.$ Moreover, we know that $\rho = \tau\cap\tau'\in\Sigma_{k-1}$ and
    \begin{equation}\label{theta}
        0 = \widehat\delta\widehat\delta1_{\rho}(\sigma) = \theta(\rho,\tau)\theta(\tau,\sigma)+\theta(\rho,\tau')\theta(\tau',\sigma).
    \end{equation}
    Consequently, we get
    \begin{align*}
        R_{k-1}\phi_v(\tau') &=\sum_{\rho\subseteq\tau',\rho\in\Sigma_{k-1}} \theta(\rho,\tau')\phi_v(\rho)\\
        &= \sum_{\rho\subseteq\tau',\rho\in\Sigma_{k-1}} \theta(\rho,\tau')\theta(\rho,v\rho)\phi(v\rho)\\
        &= \sum_{\tau\subseteq\sigma,\tau\in\Sigma_k}\theta(\tau\cap\tau',\tau')\theta(\tau\cap\tau',\tau)\phi(\tau)\\
        &= -\sum_{\tau\subseteq\sigma,\tau\in\Sigma_k}\theta(\tau',\sigma)\theta(\tau,\sigma)\phi(\tau)\\
        &= -\theta(\tau',\sigma) R_k\phi(\sigma).
    \end{align*}
 In the first and last equalities we use $\iota_{k}\delta_{k-1}\phi(\tau') = 0=\iota_{k+1}\delta_{k}\phi(\sigma)$ as $\tau',\sigma\notin \Sigma$, in the  third equality we identify $\rho$ with $\tau\cap\tau'$ and in the fourth equality we use \eqref{theta}. Taking absolute values finishes the proof.
\end{proof}
\begin{lemma}\label{errorsum}
    For every $k\geq 1$ and $\phi\in C(\Sigma_k)$, we have
    \begin{equation*}
        \beta_k\Vert R_k\phi\Vert^2\geq(k+1)\Vert R_k\phi\Vert^2-\sum_{v\in V}\Vert R_{k-1}\phi_v\Vert^2\geq \alpha_k\Vert R_k\phi\Vert^2.
    \end{equation*}
\end{lemma}
\begin{proof}
    Any $\sigma\in \K_{k+1}\backslash\Sigma_{k+1}$ such that $R_k\phi(\sigma)\neq 0$ must satisfy $\sigma\in A_k.$ Therefore, we get
    \begin{align*}
        &(k+1)\Vert R_k\phi\Vert^2-\sum_{v\in V}\Vert R_{k-1}\phi_v\Vert^2\\
        & = (k+1)\sum_{\sigma\in \K_{k+1}\backslash\Sigma_{k+1} } \vert R_k\phi(\sigma)\vert^2-\sum_{v\in V}\sum_{\tau'\in \K_k\backslash\Sigma_k}\vert R_{k-1}\phi_v(\tau')\vert^2\\
        &= \sum_{\sigma\in \K_{k+1}\backslash\Sigma_{k+1}} \vert R_k\phi(\sigma)\vert^2\left(k+1-\#\lbrace (v,\tau')\mid v\in V,\tau'\in \K_k\backslash\Sigma_k, \sigma = v\tau'\rbrace\right)\\
        &= \sum_{\sigma\in \K_{k+1}\backslash\Sigma_{k+1}} \vert R_k\phi(\sigma)\vert^2\left(\#\lbrace\tau\in\Sigma_k\mid\tau\subseteq\sigma\rbrace-1\right)
    \end{align*}
    The second equation follows from Lemma~\ref{R_k-R_k-1}.
    The third equation follows from the fact that  $\sigma$ has $k+2$ vertices, so there are $k+2$ pairs $(v,\tau')$ with $\sigma=v\tau'$ with say $N_\sigma$ of these have $\tau'\in\Sigma_k$, leaving $(k+2)-N_\sigma$ to be subtracted, which simplifies the bracket to $N_\sigma-1$. 
    From this the lemma follows immediately.
\end{proof}

\begin{proof}[Proof of Theorem~\ref{main}]
We start with the case $\circ = H.$\\
Let $\phi$ be an eigenfunction of $\lambda^H_k.$ We apply Lemma~\ref{nId} to $\phi$ and $\phi_v,v\in V$ and use Lemma~\ref{claims} to get the identity
\[(k+1)\left(\Vert\delta_k\phi\Vert^2+\Vert\partial_{k-1}\phi\Vert^2+\Vert R_k\phi\Vert^2\right) = \sum_{v\in V}\Vert\delta_{k-1}\phi_v\Vert^2+\Vert\partial_{k-2}\phi_v\Vert^2+\Vert R_{k-1}\phi_v\Vert^2.\]
Now, we use that $\Vert\delta_k\phi\Vert^2+\Vert\partial_{k-1}\phi\Vert^2 = \lambda^H_k\Vert\phi\Vert^2$ and $\Vert\delta_{k-1}\phi_v\Vert^2+\Vert\partial_{k-2}\phi_v\Vert^2\leq\lambda^H_{k-1}\Vert\phi_v\Vert^2$ and get
\[(k+1)\lambda^H_k\Vert\phi\Vert^2 +(k+1)\Vert R_k\phi\Vert^2-\sum_{v\in V}\Vert R_{k-1}\phi_v\Vert^2\leq\lambda^H_{k-1}\sum_{v\in V}\Vert\phi_v\Vert^2.\]
Using Lemma~\ref{errorsum} and Lemma~\ref{claims} yields
\[(k+1)\lambda^H_k\Vert\phi\Vert^2 +\alpha_k\Vert R_k\phi\Vert^2\leq(k+1)\lambda^H_{k-1}\Vert\phi\Vert^2.\]
Finally, we use that $\Vert R_k\phi\Vert^2 = (n-\lambda^H_k)\Vert\phi\Vert^2,$ which follows from Lemma~\ref{nId}. This proves the estimate for $\lambda^H_k.$\\
The proof for $\lambda^+_k$ follows by the exact same arguments. This is because an eigenfunction $\phi$ of $\lambda^+_k$ satisfies $\Vert\partial_{k-1}\phi\Vert^2 = 0$ and $\Vert\partial_{k-2}\phi_v\Vert^2 = 0$ for all $v\in V.$ The latter is due to Lemma~\ref{claims}.\\
The proof for $\mu^H_k$ is similar to the proof for $\lambda^H_k,$ just with reversed inequalities and $\alpha_k$ exchanged for $\beta_k.$\\
The proof for $\lambda^-_k$ follows because $\lambda^-_k = \lambda^H_k$ for which the statement is already proven. We proved the identity $\lambda^-_k = \lambda^H_k$ in Corollary~\ref{conjecture}, which only relied on the statement of Theorem~\ref{main} for $\lambda^+_k$ which we already established.
  \end{proof}

As a corollary of Theorem~\ref{main}, we get the following vanishing criterion for the $k$-th reduced cohomology group, defined for $k\ge 0$ as
\[H^k(\Sigma,\mathbb{R})=\ker \delta_k / \mathrm{ran}\, \delta_{k-1}.\]
From now on $\mu^H_k = \min\sigma(\Delta^H_k)$ will always refer to the smallest eigenvalue of $\Delta^H_k.$
\begin{corollary}\label{homology}
    Let  $k-1\geq\ell\geq 0$. Suppose     \[\mu^H_\ell> n\left(1-\frac{(\ell+1)!}{(k+1)!}\prod_{j = \ell+1}^{k}(j+1-\beta_j)\right).\]
    Then, \[H^k(\Sigma,\mathbb{R)} = \lbrace0\rbrace.\]
\end{corollary}
\begin{proof}
    By the well-known discrete Hodge-Theorem, we have $H^k(\Sigma,\mathbb{R})\cong\ker\Delta^H_k$, see  \cite{Eckmann1944}. Thus, the statement follows if $\mu^H_k>0.$ Iterating Theorem~\ref{main} yields \[\mu^H_k\frac{(\ell+1)!}{(k+1)!}\prod_{i=\ell+1}^{k}(i+1-\beta_i)\geq \mu_\ell^H-n\left(1-\frac{(\ell+1)!}{(k+1)!}\prod_{i = \ell+1}^k(i+1-\beta_i)\right),\]
    which implies the statement.\end{proof}

\begin{remark} Corollary~\ref{homology} generalizes \cite[Theorem~1.2]{AharoniBergerMeshulam} for flag complexes to general simplicial complexes. Indeed, as mentioned in the remark after Corollary~\ref{conjecture}, we find for flag complexes that $\beta_i\leq 1$ for all $i\geq 1.$ Thus, Corollary~\ref{homology} implies here, that $H^k(\Sigma,\mathbb{R})=\lbrace 0\rbrace$ if \[\mu^H_0> n\left(1-\frac{1}{(k+1)!}\prod_{i = 1}^k(i+1-1)\right) = n\frac{k}{k+1}.\] This is precisely the statement from \cite{AharoniBergerMeshulam}.
\end{remark} 

Theorem~\ref{main} and Corollary~\ref{homology} also improve the results \cite[Theorem~1.2 and Theorem~1.3]{lewgaps}. To make this more obvious, we prove the following Lemma.
\begin{lemma}\label{lew}
    Let $d$ be the maximal dimension of a \textit{missing face} in the sense of \cite{lewgaps}. In other words, $d = \max\lbrace k\mid\beta_{k-1} = k\rbrace.$
    We have \[\beta_k\leq d\]
    for all $k\geq 0.$
\end{lemma}
\begin{proof}
    We prove the statement by contradiction. Let $k_0$ be the smallest integer such that $\beta_{k_0}\geq d+1.$ This means, in particular, that $k_0\geq d.$ From the definition of $\beta_{k_0}$ we get the existence of $\sigma\in \mathcal{K}_{k_0+1}\backslash\Sigma_{k_0+1}$ and $\tau_1,\ldots,\tau_{d+2}\in\Sigma_{k_0}$ pairwise distinct, such that $\tau_i\subseteq\sigma$, $i = 1,\ldots,d+2.$ Moreover, we find $\tau'\in\mathcal{K}_{k_0}\backslash\Sigma_{k_0}$ with $\tau'\subseteq\sigma.$ If there were no such $\tau'$, then we would conclude that $\beta_{k_0} = k_0+1,$ which contradicts the maximality of $d.$\\
    Now observe that $\rho_i = \tau'\cap\tau_i\in\Sigma_{k_0-1}$, $i = 1,\ldots,d+2$ and that these also are pairwise distinct. Since $\rho_i\subseteq\tau',$ we conclude that $\beta_{k_0-1}\geq d+1,$ which contradicts the minimality of $k_0.$
\end{proof}
Lemma~\ref{lew} shows that the estimate for $\mu^H_k$ in Theorem~\ref{main} is stronger than \cite[Theorem~1.2]{lewgaps} and that the criterion in Corollary~\ref{homology} is weaker than the one in \cite[Theorem~1.3]{lewgaps}.
\section{More lower bounds on the spectrum of $\Delta^H_k$ }\label{morebounds}
In this section, we record some more estimates on the spectrum. Those estimates are rather basic but seem not to be recorded in the literature and will be used in the  section below on the Alexander dual. They are independent of the results in the previous sections.\\
We start with an estimate in terms of the number of free faces. A simplex $\rho\in\Sigma$ is called \textit{free} or a \textit{free face} if it has exactly one coface.

\begin{proposition}[Free faces]
    Let $d = \dim(\Sigma)$ and assume that every $\tau\in\Sigma_d$ has at least $m$ free faces. Then \[\mu^H_d\geq m.\]
\end{proposition}
\begin{proof}
    As $\Sigma_{d+1} = \emptyset,$ it follows that $\Delta^H_d = \Delta^-_d.$ Thus, for any $\phi\in\ell^2(\Sigma_d)$, we get
\begin{align*}    \Vert\partial_{d-1}\phi\Vert^2 = \sum_{\rho\in\Sigma_{d-1}}\Big\vert\sum_{\tau\succ\rho}\theta(\rho,\tau)\phi(\tau)\Big\vert^2\geq\sum_{\substack{\rho\in\Sigma_{d-1}\\ \rho~ \text{free}}}\sum_{\tau\succ\rho}\vert\phi(\tau)\vert^2\geq m\Vert\phi\Vert^2,
\end{align*}
where the first inequality follows from the fact that for a free face $\rho$ there is exactly one $\tau\succ\rho$ and the second inequality follows from the assumption that every $\tau\in\Sigma_d$ has at least $m$ free faces.
Thus, the statement follows.
\end{proof}
The second estimate comes from a lower bound on the Forman curvature $c^H,$ which is itself a well-studied object in the literature; see e.g. \cite{BK,F,JM,Leal,Saucan, Sreejith}. In \cite{BK} we showed that $\Vert\delta_k\phi\Vert^2+\Vert\partial_{k-1}\phi\Vert^2$ can be represented by the energy of a signed Schrödinger operator. More precisely, one has
\[\Vert\delta_k\phi\Vert^2+\Vert\partial_{k-1}\phi\Vert^2 = \frac{1}{2}\sum_{\tau,\tau'\in\Sigma_k}b^H(\tau,\tau')\vert\phi(\tau)-o^H(\tau,\tau')\phi(\tau')\vert^2+\sum_{\tau\in\Sigma_k}c^H(\tau)\vert\phi(\tau)\vert^2,\]
where the terms $b^H(\tau,\tau')$ are non-negative, $o^H(\tau,\tau')\in\lbrace\pm 1\rbrace$ and \[c^H(\tau) = 2(k+1)+(k+2)\#\lbrace\sigma\succ\tau\rbrace-\sum_{\rho\prec\tau}\#\lbrace \tau'\succ \rho\rbrace,\]
see \cite{BK,lewfaces}. This can be interpreted as a discrete version of the Bochner-Weitzenböck formula.\\
From this, we immediately get the estimate
\[\mu^H_k\geq\min_{\tau\in\Sigma_k}c^H(\tau).\] In general, this estimate can be trivial because $c^H$ can assume negative values while the spectrum is non-negative. However, under additional assumptions, it can also lead to interesting estimates on the spectrum. For example, in \cite{lewfaces} Lew uses this to derive a lower bound on the spectrum in terms of large missing faces in $\Sigma.$ The following estimate also concerns missing faces, though in a different way than Lew's estimate.\\
For a given $\tau\in\Sigma_k,$ let \[T_k(\tau) = \lbrace\sigma\in \K_{k+1}\backslash\Sigma_{k+1}\mid\tau\subseteq\sigma\rbrace\subseteq A_k.\]
In other words, $T_k(\tau)$ is the set of \textit{missing cofaces} of $\tau.$\\
We introduce the following notation
\[d_k = \min_{\tau\in\Sigma_k}\#\lbrace\sigma\succ\tau\rbrace\]
for the minimal degree of $k$-simplices.
\begin{proposition}\label{curvature}
    Assume that $\#V =n$ and let $\tau\in\Sigma_k$, $k\geq 0.$ Then
    \[c^H(\tau) = n-\sum_{\sigma\in T_k(\tau)}\#\lbrace \tau'\in\Sigma_k\mid\tau'\subseteq\sigma\rbrace.\]
In particular,
\[\mu^H_k\geq n-\max_{\tau\in\Sigma_k}\sum_{\sigma\in T_k(\tau)}\#\lbrace \tau'\in\Sigma_k\mid\tau'\subseteq\sigma\rbrace,\]
which can further be estimated by
\begin{align*}
    \mu^H_k&\geq -\beta_kn+(\beta_k+1)(k+1)+(\beta_k+1)d_k
    \\&\geq -(k+1)n+(k+2)(k+1)+(k+2)d_k.
\end{align*}

\end{proposition}
The first estimate is particularly useful if one is able to control the term $\#\lbrace \tau'\in\Sigma_k\mid\tau'\subseteq\sigma\rbrace.$ It can for example be estimated by $\beta_k+1,$ which leads to the second estimate. This term is always bounded from above by $k+2,$ but for flag complexes one even has \[\#\lbrace \tau'\in\Sigma_k\mid\tau'\subseteq\sigma\rbrace\leq 2\] because otherwise, if $\sigma$ had three or more faces in $\Sigma_k,$ those would form a clique which contradicts $\sigma\in \K_{k+1}\backslash\Sigma_{k+1}.$ The second estimate is sharper than the one obtained in \cite[Theorem~1.1]{lewfaces} about missing faces in $\Sigma,$ which follows from Lemma~\ref{lew}.
The last estimate is a generalization of the corresponding bound for graphs that was obtained by Fiedler, see \cite[3.8]{Fiedler}, that says that the second smallest eigenvalue $\mu(G)$ of the graph Laplacian $L(G)$ of a graph $G$ is bounded from below by \[\mu(G)\geq 2+2\min_{v\in V}\deg(v)-n.\] Recall that $\mu(G)$ is equal to the smallest eigenvalue of the Hodge Laplacian $\Delta^H_0$ (i.e. $k = 0$), if we interpret $G$ as a one dimensional simplicial complex.
\begin{proof}[Proof of Proposition~\ref{curvature}]

Let $\tau,\tau'\in\Sigma_k.$ We write $\tau\sim\tau'$ if both share a common face but have no common coface. This means that $\tau\cap\tau'\prec\tau,\tau'\nprec\tau\cup\tau'.$
It follows from \cite[Lemma 3.13]{BK} that 
\[c^H(\tau) = k+1+\#\lbrace\sigma\succ\tau\rbrace-\#\lbrace\tau'\sim\tau\rbrace.\]
Notice that \[\#\lbrace\tau'\sim\tau\rbrace = \sum_{\sigma\in T_k(\tau)}\sum_{\tau'\in\Sigma_k}1_{\tau'\neq\tau}1_{\tau'\subseteq\sigma} = \sum_{\sigma\in T_k(\tau)}\sum_{\tau'\in\Sigma_k}1_{\tau'\subseteq\sigma}-\#T_k(\tau).\]
Further, observe that 
\begin{equation}\label{Tk}
    \#T_k(\tau) = n-(k+1)-\#\lbrace\sigma\succ\tau\rbrace.
\end{equation}
We use this to obtain
\begin{align*}
    c^H(\tau) = n-\sum_{\sigma\in T_k(\tau)}\sum_{\tau'\in\Sigma_k}1_{\tau'\subseteq\sigma},
\end{align*}
which shows the first statement. The first estimate for $\min\sigma(\Delta^H_k)$ follows directly. The last two estimates follow from (\ref{Tk}) and the estimate $\#\lbrace \tau'\in\Sigma_k\mid\tau'\subseteq\sigma\rbrace\leq \beta_k+1\leq k+2.$
\end{proof}

\section{Alexander Duality}\label{alex}

In this section, we apply the results of the previous sections to the Alexander dual $\Sigma^\ast$ of a given simplicial complex $\Sigma$ and then reinterpret the results in terms of $\Sigma.$ That way, we obtain additional bounds on the spectrum and another vanishing criterion for cohomology classes.\\
Recall that the \textit{Alexander dual} of $\Sigma$ is defined as 
\[\Sigma^* = \lbrace\tau\subseteq V\mid \tau^c\notin\Sigma\rbrace\] and that this is again a simplicial complex.
Duval and Reiner showed in \cite[Corollary~4.7]{DuvRei} that the Hodge Laplacians $\Delta^H_k(\Sigma)$ and $\Delta^H_{n-k-3}(\Sigma^*)$ of $\Sigma$ and $\Sigma^*$ satisfy \[\sigma(\Delta^H_k(\Sigma))\backslash\lbrace n\rbrace = \sigma(\Delta^H_{n-k-3}(\Sigma^*))\backslash\lbrace n\rbrace.\] Notice that this implies \[\min\sigma(\Delta^H_k(\Sigma)) =\min \sigma(\Delta^H_{n-k-3}(\Sigma^*)),\] as both spectra are bounded above by $n.$ We use this fact to derive more bounds on the spectrum.
The following theorem is a consequence of Theorem~\ref{main} applied to $\Sigma^*.$

\begin{theorem}\label{alexestimate}
    Let $\Sigma$ be a simplicial complex with $n = \#V.$ Let $0\leq k\leq\dim(\Sigma)-1$ and assume that $\dim(\Sigma)\leq n-3.$ Then
    \[d_k\mu^H_k+(n-k-2-d_k)n\geq(n-k-2)\mu^H_{k+1}.\]
\end{theorem}
\begin{proof}
    Notice that the statement holds if $\mu_k^H = n,$ as $\mu^H_{k+1}\leq n$ is always true. Therefore, we may assume that $\Sigma_{k+1}$ is not complete, because otherwise $\mu^H_k = n.$ This guarantees that $\Sigma^*_{k'}$ for  $k' = n-k-3\geq 1$ is not empty and, therefore, $d_k\leq n-k-2$.\\
    Let   $\widehat\mu^H_{k'}$ and $\widehat\mu^H_{k'-1}$ be the smallest eigenvalues of $\Delta^H_{k'}(\Sigma^*)$ and $\Delta^H_{k'-1}(\Sigma^*).$
    Since $\dim(\Sigma)\leq n-3,$ we find that $\Sigma^*$ and $\Sigma$ have the same set of vertices $V.$
    By Theorem~\ref{main} applied to $\Sigma^*,$ we get     \[(k'+1-\beta^*_{k'})\widehat\mu^H_{k'}+\beta^*_{k'}n\geq(k'+1)\widehat\mu^H_{k'-1},\]
    where $\beta^*_{k'} = \beta_{k'}(\Sigma^*).$ \\
 The rest of the proof is just a translation of the terms involved. As mentioned above, we have $\widehat\mu^H_{k'} = \mu^H_k$ and  $\widehat\mu^H_{k'-1} = \mu^H_{k+1}.$ For the term $\beta^*_{k'}$ we find
 \begin{align*}
     \beta^*_{k'} &= \max_{\tau^c\in \K_{k'+1}\backslash\Sigma^*_{k'+1}}\#\lbrace\sigma^c\in\Sigma^*_{k'}\mid\sigma^c\subseteq\tau^c\rbrace-1\\
     & = \max_{\tau\in\Sigma_k}\#\lbrace \sigma\in A_k\mid\tau\subseteq\sigma\rbrace-1\\
     & = n-k-2-d_k. \end{align*}
      This finishes the proof. \end{proof}
Theorem~\ref{alexestimate} gives another vanishing criterion for $H^k(\Sigma,\mathbb{R}),$ similar to Corollary~\ref{homology}. However, this time it concerns the smallest eigenvalue in the upper dimensions of $\Sigma.$
\begin{corollary}
    Let $d = \dim(\Sigma)$ and $d\geq\ell\geq k+1\geq 1.$ Assume that $\#V = n\geq d+3$ and 
    \[\mu^H_\ell> n\left(1-\prod_{i= 0}^{\ell-k-1}\frac{d_{k+i}}{(n-k-i-2)}\right).\]
    Then \[H^k(\Sigma,\mathbb{R}) = \lbrace 0\rbrace.\]
\end{corollary}
\begin{proof}
    By using Theorem~\ref{alexestimate} inductively, one shows that for all $1\leq \ell'\leq d-k$
    \[\mu^H_k\prod_{i = 0}^{\ell'-1}\frac{d_{k+i}}{(n-k-i-2)}\geq \mu^H_{k+\ell'}-n\left(1-\prod_{i = 0}^{\ell'-1}\frac{d_{k+i}}{(n-k-i-2)}\right).\] 
    For $\ell = \ell'+k$, we conclude that $\mu^H_k >0.$ Thus, $H^k(\Sigma,\mathbb{R})\cong\ker\Delta^H_k = \lbrace0\rbrace.$
\end{proof}
Finally, we derive another lower bound on $\sigma(\Delta_k^H)$ by considering the curvature $c^H$ of the Alexander dual $\Sigma^*,$ instead of the curvature of $\Sigma$ itself.
\begin{proposition}
   Assume that $n = \#V$ and $0\leq k\leq\dim(\Sigma)\leq n-2.$ Then,
    \[\mu^H_k\geq n-\max_{\sigma\in \K_{k+1}\backslash\Sigma_{k+1}}\sum_{\tau\subseteq\sigma,\tau\in\Sigma_k}\left(n-k-1-\#\lbrace\sigma'\succ\tau\rbrace\right).\]
\end{proposition}
\begin{proof}
    This is essentially an application of Proposition~\ref{curvature} to $\Sigma^*$ and using the identity between the spectra from above. The assumption $\dim(\Sigma)\leq n-2$ ensures that $\Sigma$ and $\Sigma^*$ have the same set of vertices $V.$\\
    Consider some $\sigma\in \K_{k+1}\backslash\Sigma_{k+1}.$ Then $\sigma^c = V\backslash\sigma\in\Sigma^*_{n-k-3}.$ Let 
    \[\tau^c\in T^*_{n-k-3}(\sigma^c) = \lbrace\tilde\tau^c\in \K_{n-k-2}\backslash\Sigma^*_{n-k-2}\mid\sigma^c\subseteq\tilde\tau^c\rbrace. \]
    Then, $\tau\subseteq\sigma$ and $\tau\in\Sigma_k.$
    Now, let $c^H_*(\sigma^c)$ be the Forman curvature with respect to $\Sigma^*.$ Then, according to Proposition~\ref{curvature}
\begin{align*}
    c^H_*(\sigma^c) &= n-\sum_{\tau\subseteq\sigma,
\tau\in\Sigma_k}\#\lbrace\tilde\sigma^c\in\Sigma^*_{n-k-3}\mid\tilde\sigma^c\subseteq\tau^c\rbrace \\
&= n-\sum_{\tau\subseteq\sigma,\tau\in\Sigma_k}\#\lbrace\tilde\sigma\in \K_{k+1}\backslash\Sigma_{k+1}\mid\tau\subseteq\tilde\sigma\rbrace\\
& = n-\sum_{\tau\subseteq\sigma,\tau\in\Sigma_k}(n-k-1-\#\lbrace\tilde\sigma\in\Sigma_{k+1}\mid\tau\subseteq\tilde\sigma\rbrace).
\end{align*}

The proof follows because \[\min_{\sigma^c\in\Sigma^*_{n-k-3}}c^H_*(\sigma^c)\leq\min\sigma(\Delta^H_{n-k-3}(\Sigma^*)) = \min\sigma(\Delta^H_k(\Sigma)).\hfill\qedhere\]
\end{proof}
\textbf{Acknowledgements.} The authors appreciate the financial support of the DFG and MK enjoyed the hospitality of the IIAS Jerusalem. PB thanks Lior Tenenbaum for discussions about Alexander duality and other topics relevant to the paper.

\bibliographystyle{abbrv}
\bibliography{literature}

\end{document}